\documentclass[12pt]{amsart}

\usepackage{amssymb,amsmath,graphicx,mathrsfs,mathabx,tikz,url}

\usepackage{verbatim}
\usepackage{amsrefs}

\usepackage{epsfig,verbatim}
\usepackage{subfigure}
\usepackage{enumerate}

\usepackage{srcltx}

\usepackage{footnote}
\usepackage{multicol}
\usepackage{booktabs}
\usepackage{pifont}
\usepackage[flushleft]{threeparttable}
\usepackage[font=small,labelfont=bf]{caption}
\usepackage{amssymb}
\usepackage{bm}
\usepackage[centertags]{amsmath}
\usepackage{amsfonts}
\usepackage{amsthm}
\usepackage{mathrsfs}
\usepackage{xcolor}
\usepackage{graphicx}
\usepackage{comment}

\theoremstyle{definition}
\newtheorem{thm}{Theorem}[section]
\newtheorem{dfn}[thm]{Definition}
\newtheorem{lem}[thm]{Lemma}
\newtheorem{prp}[thm]{Proposition}

\newtheorem{qtn}[thm]{Question}

\newtheorem*{thm*}{Main Theorem}

\newtheorem*{cor*}{Corollary}

\newtheorem*{prp*}{Proposition}

\newtheorem*{rmk*}{Remark}

 \newcommand \N{\mathbb N}
 
 \newcommand \R{\mathbb R}

 \newcommand \eps{\varepsilon}

\newcommand{\co}{\mathrm{c}_0}
\newcommand{\coo}{\mathrm{c}_{00}}

\newcommand \beq{\begin{eqnarray*}}
\newcommand \eeq{\end{eqnarray*}}

\newcommand{\norm}[1]{\left\Vert #1\right\Vert}

\numberwithin{equation}{section}

\begin{document}

\title[Fixed point property and affine bi-Lipschitz maps]{Selection type results and fixed point property for affine bi-Lipschitz maps}

\maketitle

\author{Cleon S. Barroso\footnote{Departamento de Matem\'atica, Universidade Federal do Cear\'a, Fortaleza, CE 60455-360 Brazil,  \ Email: {\tt cleonbar@mat.ufc.br }} and 
Torrey M. Gallagher\footnote{Bucknell University, Lewisburg, PA 17837 USA  \ Email: {\tt tmg012@bucknell.edu }}}
{ }

\begin{abstract}
We obtain a refinement of a selection principle for $(\mathcal{K}, \lambda)$-wide-$(s)$ sequences in Banach spaces due to Rosenthal. This result is then used to show that if $C$ is a bounded, non-weakly compact, closed convex subset of a Banach space $X$, then there exists a Hausdorff vector topology $\tau$ on $X$ which is weaker than the weak topology, a closed, convex $\tau$-compact subset $K$ of $C$ and an affine bi-Lipschitz map $T\colon K\to K$ without fixed points.  
\end{abstract}

\keywords{Fixed point, affine, bi-Lipschitz, basis, weakly compact, coaser vector topologies, selection principles, wide-$(s)$ sequences, James's distortion theorems.}

\thanks{{\em 2010 Mathematics Subject Classification:} 47H10, 46B50, 46B15}

\section{Introduction}\label{sec:intro}
Let $(X, \norm{\cdot})$ be a Banach space.  The majority of the work done in metric fixed point theory has been concerned with determining when nonexpansive mappings defined on closed, bounded, convex sets have fixed points. Recall, if $C \subset X$ and $T: C \to X$ then we say $T$ is nonexpansive if 
\[
\norm{Tx-Ty}\leq \norm{x-y}\quad \text{ for all }\, x, y\in C.
\] 
Further, $T$ has a fixed point $x$ if $Tx=x$. However, there has been very fruitful work done that is more topological in nature when one considers fixed point properties for larger classes of mappings. For instance, Dom\'inguez Benavides, Prus and Jap\'on Pineda \cite{BPP} proved the following beautiful result in 2004.

\begin{thm}\label{thm:1} Let $C$ be a bounded, closed convex subset of $X$. Then $C$ is weakly compact if and only if for every closed convex set $K\subset C$, every continuous affine mapping $T\colon K\to K$ has a fixed point. 
\end{thm}

Very recently, Barroso and Ferreira \cite{BF} strengthened this result so that now the class of maps is more regular.  

\begin{thm}\label{thm:2} Let $C$ be a bounded, closed convex subset of $X$. Then $C$ is weakly compact if and only if  for every closed convex set $K\subset C$, every affine bi-Lipschitz mapping $T\colon C\to C$ has a fixed point. 
\end{thm}

In passing we mention that it seems to be an open problem (cf. \cite{BF, BPP}) if the family of maps entailed in these statements can be constrained to the class of affine uniformly Lipschitzian maps; that is, affine mappings for which there exists a constant $k$ for which 
\[
\norm{T^n(x) - T^n(y)} \leq k\norm{x-y}
\]
for all $n$ and for all $x,y \in C$.  

\smallskip 

Given a topology $\tau$ on $X$ and a class of mappings $\Sigma$ defined on $\tau$-closed, convex subsets of $X$, we say $X$ has the $\tau$ fixed point property for the class $\Sigma$ (abbreviated as $\tau$-FPP or simply FPP) if every mapping in $\Sigma$ defined on a $\tau$-compact, convex set has a fixed point. Towards the so-called $\tau$-FPP, Jap\'on Pineda \cite{JP} proved the following.

\begin{thm} Let $X$ be a Banach space containing an isomorphic copy of $\co$. Then $X$ admits a Hausdorff vector topology $\tau$ which is weaker than the weak topology and is such that $(X,\tau)$ fails to have the $\tau$-FPP for asymptotically nonexpansive mappings. 
\end{thm}

It is interesting to note that the containment of copies of $\co$ always gives rises to linear topologies which are weaker than $\sigma(X,X^*)$ (cf. \cite[p.186]{JP}). It turns out that if $C$ is not weakly compact, then similar arguments can be used to build such topologies in general spaces. 

The goal of this paper is to study the following natural question: 
\begin{qtn} Is there a bi-Lipschitz counterpart of Jap\'on Pineda's result for general Banach spaces.  
\end{qtn}

This is of course much more easier to answer when the underlying space has more structure. This has been the key point behind several works concerning nonexpansive maps. Here we give a starting point towards positive answers. Let $\mathcal{T}_w$ denote the family of all Hausdorff locally convex topologies that are weaker than the weak topology on $X$. For a convex set $C\subset X$, $\mathcal{B}(C)$ stands for the family of nonempty bounded, closed convex subsets of $C$.  Let $\mathcal{K}_\tau(C)$ ($\tau\in \mathcal{T}_w$) be the subfamily of $\mathcal{B}(C)$ consisting of sets that are $\tau$-compact. Motived by the study of the FPP under weaker topologies (e.g., $\tau$-FPP, cf. \cite{B-GF-JP}) we define:

\begin{dfn}  A set $C\in \mathcal{B}(X)$ is said to have the $(w,\mathcal{G})$-$FPP$ for a class $\mathcal{M}$ of maps if whenever $\tau\in \mathcal{T}_w$ and $K\in \mathcal{K}_\tau(C)$, then every $K$-invariant map belonging to $\mathcal{M}$ has a fixed point. 
\end{dfn}

A fruitful strategy to establish FPP is to try recognizing what is the internal structure of non-weakly compact sets. A prime example of this is witnessed in  \cite{JP}. In fact, in order to prove her result, Jap\'on Pineda considered the following $\sigma$-convex conical hull 
\[
K^{+}_\sigma\big( \{ u_n\}\big)=\Bigg\{ \sum_{i=1}^\infty t_i u_i \colon\,\text{ each }\, 0\leq t_i \leq 1,\quad \sum_{i=1}^\infty t_i\leq 1   \Bigg\}
\]
associated to a well-structured basic sequence $(u_n)$. She then dealt with a standard affine mapping given by
\[
T\left( \sum_{i=1}^\infty t_i u_i\right) = \left( 1- \sum_{i=1}^\infty t_i\right) u_1 + \sum_{i=1}^\infty t_i u_{i+1},
\]
and used the special structure of $(u_n)$ to prove that $K^{+}_\sigma\big( \{ u_n\}\big)$ is $\tau$-compact and $T$ is asymptotically nonexpansive and fixed-point-free. 

Intuitively, one can start with a non-weakly compact set $C\in \mathcal{B}(X)$, select a basic sequence $(u_n)\subset C$ and then trying to work directly with the set $K^{+}_\sigma\big( \{ u_n\}\big)$. Naturally, one may assume that $0\in C$. Given then the abstract nature of $X$, it is quite natural to expect that at least Lipschitz regularity for such maps can be gained. A motivation for this comes from the work of Lin and Sternfeld \cite{LS} which describes norm-compactness in terms of the FPP for Lipschitz maps. However, trying to do that seems to be nontrivial.  On the other hand, if one tries to replicate the approach of \cite{BF} then a subtle difficulty arises. Namely, the set $K$ and the map $T$ considered there are completely different and, in fact, there is no guarantee that $K$ should be compact under some weaker topology.  Thus, we need a different approach to obtain regularity.  Our first main result is:

\begin{thm*}\label{thm:1} Let $X$ be a Banach space. Then $C\in \mathcal{B}(X)$ is weakly compact if and only if $C$ has the $(w,\mathcal{G})$-$FPP$ for affine bi-Lipschitz maps. 
\end{thm*}

The proof of this result relies on a quantified version of a selection principle of Rosenthal \cite{Ro2} (cf. Lemma \ref{lem:1}). The motivation for this tool lies in our desire to detect wide-(s)  sequences (defined below) whose initial terms satisfy some prescribed conditions. As it might happen however, none of their terms must a priori satisfy such conditions. To fix this problem, the idea is to add new terms and try to capture large indices with stable constants $(\mathcal{K},\lambda)$. In Section \ref{sec:rose}, we will lay the foundation for proving the aforementioned selection principle. The proof of Main Theorem will be delivered in Section \ref{sec:proof}. 

The idea of refining selection results has been very useful in fixed point theory. In \cite[Theorem 10]{DLT} Dowling, Lennard and Turett proved that if $X$ contains an isomorphic copy of $\co$, then there exist $C\in \mathcal{B}(X)$ and an affine mapping $T\colon C\to C$ being {\it almost-isometric}, asymptotically nonexpansive and fixed-point free. This result was based on a refinement of the well-known James's distortion theorems \cite{J}. 


\medskip 
\noindent
{\bf Acknowledgement:} The first author gratefully acknowledges financial support by
FUNCAP/CNPq/PRONEX Grant 00068.01.00/15. A great part of this work was done when the first author was visiting  Texas A$\&$M University and University Central of Florida during Summer 2018. He wishes to thank Professors Thomas Schlumprecht and Eduardo Teixeira for their kind invitation and the corresponding Mathematics Department for the hospitality and very nice working environment.


\section{Preliminaries}\label{sec:prelim} 
As we have indicated, our approach for proving this result is entirely based on basic sequences techniques, most of them extracted from \cite{AK}. We shall next review some basic notions and establish terminology. The closed unit ball and closed unit sphere of  $X$ will respectively be denoted by $B_X$ and $S_X$. 

A sequence $(x_n)$ is called a basic sequence if it is a Schauder basis for its closed linear span $[x_n]$. It is well-known that $(x_n)$ is basic if and only if there is a constant $\mathcal{K}\geq 1$, the basic constant of $(x_n)$, such that
\[
\norm{\sum_{n=1}^m a_i x_i} \leq \mathcal{K} \norm{\sum_{i=1}^n a_i x_i }\quad \text{ whenever }\,\, (a_i)\in \coo \,\,\text{ and }\,\, n< m.
\]
We say that $(x_n)$ is an $\ell_1$-sequence if it is equivalent to the unit vector basis of $\ell_1$ (so $[x_n]$ is isomorphic to $\ell_1$). A subspace $G$ of the dual space $X^*$ is called an $\eta$-norming ($0<\eta \leq 1$) set for $X$ if
\[
\norm{x}_G:= \sup_{\varphi \in B_G}| \varphi(x)|  
\]
is a norm on $X$ and $\eta \norm{ x}\leq \norm{ x}_G\leq \norm{ x}$ for all $x\in X$. Given an infinite subset $N$ of $\mathbb{N}$, we shall denote by $[N]$ the set of all infinite subsets of $N$. Let us recall some common terminology.  A sequence $(x_n)$ in $X$ is called \textit{seminormalized} if there are positive constants $A$ and $B$ for which $A \leq \norm{x_n} \leq B$ for all $n$. Let $X$ and $Y$ be Banach spaces. If $(x_n)\subset X$ and $(y_n)\subset Y$ are sequences, we say $(x_n)$ \textit{dominates} $(y_n)$ if 
\[
\sum_{n=1}^\infty \alpha_n x_n  \text{ converges implies that }  \sum_{n=1}^\infty \alpha_n y_n  \text{ converges as well}. 
\] 
Let us specify the notion of $\lambda$-dominance. Let $(x_n)\subset X$ and $(y_n)\subset Y$ be as above. Given a constant $\lambda>0$, we say the basis $(x_n)$ $\lambda$-dominates the basis $(y_n)$ if $(x_n)$ dominates $(y_n)$ and, when $\sum_n a_n x_n$ converges, it follows that
\[
\norm{\sum_{n=1}^\infty a_n y_n} \leq \lambda  \norm{\sum_{n=1}^\infty a_n x_n}.
\]
\medskip 
\noindent Recall also the \textit{summing basis} of $c_0$: $(\sigma_n) = \left( \sum_{k=1}^n e_k \right)$, where \[
e_k = (\underbrace{0,\ldots, 0}_{k-1}, 1, 0, \ldots)
\] 
is the usual basis in $c_0$.   According to \cite{Ro2}, a sequence $(x_n)$ is called wide-$(s)$ if it is basic and dominates the summing basis of $\co$; that is, $(x_n)$ $\lambda$-dominates the summing basis of $\co$ for some $\lambda>0$. Equivalently, if convergence of $\sum_n \alpha_n x_n$ implies convergence of $\sum_n \alpha_n$. 

\medskip 
In the sequel we will use the terminology $(\mathcal{K},\lambda)$-wide-$(s)$ to refer those wide-$(s)$ sequences that are $\mathcal{K}$-basic and $\lambda$-dominate the summing basis of $\co$. This notion was introduced by Rosenthal \cite{Ro2}. It is also closely related to his $\co$-counterpart \cite{Ro1} of the well-knwon $\ell_1$ theorem. It is worth to point out that Barroso and Ferreira proved in \cite[Theorem 2.9]{BF} that in any Banach space $X$, non-weakly compact members of $\mathcal{B}(X)$ do contain wide-$(s)$ sequences that are equivalent to some of its non-trivial convex combinations. This kind of results are fundamental in the construction of fixed-point free affine bi-Lipschitz maps. As we will see, it also allows us to obtain the failure of the $(w,\mathcal{G})$-$FPP$ for affine bi-Lipschitz maps acting on non-weakly compact members of $\mathcal{B}(X)$.



\section{A Rosenthal selection type principle}\label{sec:rose}

The main result of this section can be stated as follows. 

\begin{prp}[Rosenthal's type selection result]\label{prop:1sec3} Let $X$ be a non-reflexive Banach space. Suppose $(y_n)$ is a bounded sequence in $X$ with a $w^*$-cluster point $\Psi$ in $X^{**}\setminus\, X$. Let $\varepsilon\in (0,1)$,
\[
d={\rm{dist}}(\Psi, [y_n]_{n\in \mathbb{N}}),\quad \mathcal{K}=\frac{\| \Psi\| + d}{d}+\varepsilon\quad\text{and}\quad\lambda=\frac{1}{d}+ \varepsilon. 
\]
Then there exists a set $N\in[\N]$ such that:
\begin{itemize}
\item[(i)] $(y_n)_{n\in N}$ is $(\mathcal{K},\lambda)$-wide-$(s)$. 
\item[(ii)]  For every vector $u\in X\setminus\{ 0\}$ there is a set $M\in[N]$ such that $\{ u, y_n\}_{n\in M}$ is $(\mathcal{K},\lambda)$-wide-$(s)$. 
\end{itemize}
\end{prp}

 For the proof of (ii) we need an auxiliary lemma which resembles famous Kadec-Pe\l czy\'nski's selection principle.

\begin{lem}\label{lem:1} Let $X$ be a Banach space and $(y_n)$ a seminormalized sequence in $X$. Assume that $G\subset X^*$ is a $\eta$-norming for $X$ with $g(y_n)\to 0$ for all $g\in G$. Then for every $0<\varepsilon<1$ there exists an infinite set $N\in [\N$] satisfying:
\begin{itemize}
\item[(a)] $(y_n)_{n\in N}$ is $(1+\varepsilon)/\eta$-basic.
\item[(b)]  For every vector $u\in X\setminus\{ 0\}$  there exists $M\in[N]$ such that $\{ u, y_n\}_{n\in M}$ is $(1+\varepsilon)/\eta$-basic. 
\end{itemize}
\end{lem}

\begin{proof} We only need to prove (b).  Let $\eps\in (0,1)$ and fix $u \in X\setminus\{0\}$. Now choose $(\varepsilon_i)$ a sequence of positive numbers satisfying $\prod_{i=1}^\infty (1+\varepsilon_i)< (1+\varepsilon)$ and write $N=\{ n_i\}_{i=1}^\infty$. After equivalently renorming $X$, we may assume that $\eta=1$. Revisiting the proofs in \cite[Proposition and lemma]{Pel} we observe that it is enough to recover the first induction step made in Pe\l czy\'nski's lemma and the remainder of the proof will follow analogously. Let $E=[\{ u\}]$ and let $e \in S_E = \{\pm u/\norm{u}\}$.  Since $G$ is norming, there is a functional $g \in S_G$ for which $1-\eps_1/4 < |g(e)|$.  Since $g(y_n) \to 0$, we can find $m_2 \in N$ for which 
\[
|g\left(y_{m_2}\right)| < \frac{\eps_1}{8} \inf_\ell \norm{y_\ell}. 
\]
Note here that $\inf_\ell \norm{y_\ell} >0$ since $(y_n)$ is seminormalized.  Let $t \in \R$ be arbitrary.  If $|t|\geq 2/\inf_\ell\| y_\ell\|$ then 
\[
\begin{split}
\norm{ e + ty_{m_2}} &\geq | t| \norm{y_{m_2}} - \| e\| \\
&\geq 2 -1 \geq \frac{\| e\| } {1+\varepsilon_1}.
\end{split}
\]
Otherwise if $| t| <  2/\inf_\ell\| y_\ell\|$ then we have, since $G$ is norming,
\[
\begin{split}
\norm{ e + ty_{m_2}} &\geq | g(e+ty_{m_2})|\\
&\geq  |g(e)| - |t||g(y_{m_2})|\\
&> 1-\frac{\eps_1}{4}  -  \frac{2\eps_1}{8} = 1- \frac{\eps_1}{2} \geq \frac{\norm{e}}{1+\eps_1}.
\end{split}
\]
In any case we have proved that $\| e\| \leq (1+\varepsilon_1)\| e + t y_{m_2}\|$ for all $t\in \mathbb{R}$.

Now we will comment on the inductive step; again, for further details, the reader should reference \cite{Pel}.  Suppose we have already chosen $y_{m_2},\ldots,y_{m_k}$.  We now let $E = [\{ u, y_{m_2},\ldots,y_{m_k}\}]$ and note that $E$ is finite dimensional.  Since $E$ is finite dimensional, $S_E$ is compact and we may thus take an $\varepsilon_{k+2}/6$-net $\{ \vartheta_i\}_{i=1}^m$ for $S_E$. Fix any $e\in S_E$. We may choose $\mu\in \{ 1,\dots, m\}$ so that $\| e - \vartheta_\mu\| <\varepsilon_{k+1}/6$. For each $i$ there is a functional $g_i\in S_G$ with $1-\varepsilon_{k+1}/6< | g_i(\vartheta_i)|$. Since $G(y_n)\to 0$ there is $m_{k+1}\in N$ so that 
\[
| g_i(y_{m_{k+1}})| <\varepsilon_{k+1} \inf_\ell \| y_\ell\|/12\quad\text{for all }i=1,\dots, m.
\]
With these ingredients, the remainder of the argument follows analogously to the above.
\end{proof}

Now we will prove Proposition \ref{prop:1sec3}.

\begin{proof} (i): This follows immediately from \cite[(1)-Theorem 12]{Ro2}. 

(ii): Let $d, \mathcal{K}$ and $\lambda$ be as in the statement of Proposition \ref{prop:1sec3} and let $u \in X\setminus\{0\}$ be fixed. Even though the idea of the proof of (ii) is essentially the same as that given in \cite[Theorem 12]{Ro2}, Lemma \ref{lem:1} plays an important role here, as it can be used to circumvent possible unstable constants. Indeed, let $N\in [\mathbb{N}]$ be as in assertion (i) and set
\[
\eta = \frac{ d}{ \| \Psi\| + d}\quad \text{ and }\quad \varepsilon_0=\frac{\varepsilon \eta}{ \mathcal{K}}.
\]
It is easy to see that $\varepsilon_0< \eta/2$. Further,  by Lemma 14 in \cite{Ro2} the set $\Psi^\perp:=\{ x^*\in X^*\colon \Psi( x^*)=0\}$ $\eta$\,-norms $X$.
Let $\tilde{\varepsilon}_0=\frac{\eta}{\eta - \varepsilon_0} -1$ and note that $0<\tilde{\varepsilon}_0<1$. Moreover, we have
\[
\frac{1 + \tilde{\varepsilon}_0}{\eta}=\frac{1 }{\eta - \varepsilon_0}=\mathcal{K}.
\]
Therefore, Lemma \ref{lem:1}-(b) gives $\mathcal{O}\in [N]$ so that $\{ u, y_n\}_{n\in \mathcal{O}}$ is still $\mathcal{K}$-basic. Let now $\delta>0$ to be chosen later. For this $\delta$ we shall select  inductively a set $\{m_i\}_{i=1}^\infty \in  [\mathcal{O}]$ and functionals $(\phi_i)_{i\in\mathbb{N}}$ in $X^*$ satisfying the following conditions for all $n$:
\begin{itemize}
\item[(1)]\, $\| \phi_i\| \leq \dfrac{1}{d} + \delta$.
\item[(2)]\, $\phi_i(u)=0$ for all $1\leq i\leq n$.
\item[(3)]\, $\phi_i ( y_{m_j})=0$ for $1\leq j< i\leq n$.
\item[(4)]\, $\Psi( \phi_i)=1$ for all $1\leq i\leq n$.
\item[(5)]\, $\big| \phi_i( y_{m_n} ) -1\big| < \dfrac{\delta}{2^n}$ for all $1\leq i\leq n$. 
\end{itemize}
By Hahn-Banach theorem we may find $F\in X^{***}$ so that
\begin{equation}\label{eqn:1lem2}
\| F \|=\frac{1}{d},\quad F(\Psi)=1\quad \text{ and } \quad F(x)=0\quad \text{ for all } x\in X. \tag{$\dagger$}
\end{equation}
For $n=1$ pick $\phi_1\in X^*$ so that $\Psi(\phi_1)=1$, $\phi_1(u)=0$ and $\| \phi_1\|<\frac{1}{d} + \delta$. Assume for $n>1$ that functionals $\phi_1,\dots, \phi_n$ and integers $m_1,\dots, m_{n-1}>1$ have been chosen to satisfy (1)--(4). Since (4) holds and $\Psi$ is $w^*$-cluster point of $(y_n)$, we may choose an $n_{i_\ell}\in \mathcal{O}$ large enough so that $| \phi_i( y_{n_{i_\ell}}) -1 | < \delta/2^\ell$ for all $1\leq i\leq n$. Define $m_n=n_{i_\ell}$. Since $[\{ u, y_{m_1},\dots, y_{m_n}, \Psi\}]$ is finite dimensional, by (\ref{eqn:1lem2}) we may choose $\phi_{n+1}\in X^*$ that agrees with $F$ on $[\{ u, y_{m_1},\dots, y_{m_n}, \Psi\}]$ and has norm $\| \phi_{n+1}\|< \frac{1}{d} + \delta$. This finishes the induction process. 

Let $M=\{ m_k\}_{k=1}^\infty$. Define $x_1=u$ and $x_i=y_{m_i}$ for $i\geq 2$. Now, for any $1\leq k\leq n$ and any sequence of scalars $(a_i)_{i=1}^n$, since $(x_i)$ is seminormalized and $\mathcal{K}$-basic, one can directly show that
\[
\sup_{1\leq i\leq n}| a_i| \leq \frac{2\mathcal{K}}{\inf_n\| x_n\|}\Big\| \sum_{i=1}^n a_i x_i\Big\|.
\] 
On the other hand, from (2) and (5) we have
\[
\begin{split}
\Big|\sum_{i=k}^n a_i\Big| &=\Big|\sum_{i=1}^n \phi_k( a_i x_i)  + \sum_{i=k}^n a_i\left( 1 - \phi_k(x_i)\right)\Big|\\
&\leq \left( \frac{1}{d} +\delta\right)\Big\|\sum_{i=1}^n a_ix_i\Big\| + \sup_n |a_n| \sum_{i=1}^n \frac{\delta}{2^i}.
\end{split}
\]
Combining both inequalities, we get
\[
\Big|\sum_{i=k}^n a_i\Big| \leq \Bigg( \frac{1}{d} + \delta + \frac{2\delta\mathcal{K}}{\inf_n\| x_n\|}\Bigg)\Big\| \sum_{i=1}^n a_i x_i\Big\|.
\]
If $\delta$ is chosen small enough so as to satisfy 
\[
\delta( 1 + \frac{2\mathcal{K}}{\inf_n\| x_n\|})\leq \varepsilon, 
\]
we obtain then that $(x_i)$ is $(\mathcal{K},\lambda)$-wide-$(s)$. This finishes the proof of the proposition. 
\end{proof}



\section{Proof of Main Theorem}\label{sec:proof}
If $C$ is weakly compact, the result follows directly from the classical Schauder-Tychonoff's fixed point theorem. Now assume that $C$ is not weakly compact. Then there is a sequence $(u_n)\subset C$ with no weakly convergent subsequence. We may without loss of generality assume that $0\in C$. 

It is easy to see that if $(u_n)_{n\in N}$ is an $\ell_1$-subsequence of $(u_n)$, then  $K^{+}_\sigma\big( \{ u_n\}_{n\in N}\big)$ belongs to $\mathcal{K}_\tau(C)$ for some $\tau\in \mathcal{T}_w$ and fails to have the fixed point property for affine bi-Lipschitz maps. 

Suppose therefore that $(u_n)$ does not contain $\ell_1$-subsequences. By Rosenthal's $\ell_1$-theorem $(u_n)$ has a weak Cauchy subsequence $(y_n)$. It follows that the functional $\Psi_0$ given by $\Psi_0(\varphi)=\lim_n\varphi(y_n)$, $\varphi\in X^*$, defines a $w^*$-cluster pint of $(y_n)$ in $X^{**} \backsim\, X$ with the property
\[
\Psi_0\in \bigcap_{L\in [\mathbb{N}]} {\rm{clust}}_{w^*}\big( ( y_n)_{n\in L}\big). 
\]
Let $\varepsilon\in (0,1)$ be fixed and consider the numbers
\[
d={\rm{dist}}(\Psi, [y_n]),\quad \mathcal{K}=\frac{ \| \Psi_0\| + d}{d}+\varepsilon\quad \text{ and  } \quad \lambda=\frac{1}{d} + \varepsilon. 
\]
By Proposition \ref{prop:1sec3} there exists $N\in [\mathbb{N}]$ so that $(y_n)_{n\in N}$ is $\mathcal{K}$-basic and $\lambda$-dominates the summing basis. Since $0\in C$, we may choose a non-null vector $u\in C$ so that 
\[
\| u\| < \frac{(1-\varepsilon) d^2}{16( 1+d\varepsilon )( \| \Psi_0\| + d(1+\varepsilon))}.
\]
Now, apply once again Proposition \ref{prop:1sec3} to select an infinite set $M=\{m_i\}_{i=2}^\infty\in [N]$ so that $\{ u, y_{m_i}\}_{i=2}^\infty$ is still $\mathcal{K}$-basic and $\lambda$-dominates the summing basis. Define a new sequence $(x_i)_{i\in\mathbb{N}}$ in $C$ by 
\[
x_1=u\quad \text{ and } \quad x_i=y_{m_i} \,\, \text{ for } \,\, i\geq 2. 
\]

We now proceed to build up a basic sequence $(z_n)$ which is a convex combination of $(x_i)$, and will centrally play a role in the failure of $(w, \mathcal{G})$-$FPP$ in $C$. First, fix a sequence of positive numbers $(\lambda^{(2)}_i)_{i=2}^\infty$ satisfying
\begin{itemize}
\item[(i)]\qquad $\displaystyle{\rm{diam}}(C)\sum_{i=2}^\infty \lambda^{(2)}_i<\frac{(1-\varepsilon)d^2}{16(1+d\varepsilon)(\| \Psi_0\|  + d(1+\varepsilon))}$. 
\item[(ii)]\qquad $\displaystyle1 -\sum_{i=2}^\infty \lambda^{(2)}_i\geq 1/2$. 
\end{itemize}
Then define
\[
z_1:=\alpha_1 x_1 + \sum_{i=2}^\infty \lambda^{(2)}_i x_i,\quad \alpha_1:= 1 -\sum_{i=2}^\infty \lambda^{(2)}_i\geq 1/2.
\]
Next for each $n\geq 2$ choose a sequence of positive numbers $(\lambda^{(n)}_i)_{i=n}^\infty$ from $B_{\ell_1}$, satisfying
\begin{itemize}
\item[(iii)]\, $\displaystyle\Big( 1 - \sum_{i=n}^\infty \lambda^{(n)}_i\Big)_{n\geq 2}\,$ is nondecreasing in $[\alpha_1,1)$.
\item[(iv)]\, $\displaystyle1/2 \leq 1  - \sum_{i=n}^\infty \lambda^{(n)}_i\to 1$ as $n\to \infty$, 
\item[(v)] \, and, moreover, 
\end{itemize}
\begin{equation}\label{eqn:1sec3}
\sum_{n=1}^\infty \sum_{k=n}^\infty \lambda^{(n)}_k< \frac{\varepsilon \inf_n\| u_n\|}{4\mathcal{K} \sup_n \| x_n\|}
\end{equation}
We these numbers in hands, we can define $(z_n)$ by  
\begin{equation}\label{eqn:zn}
z_n:=\left( 1 - \sum_{k=n+1}^\infty \lambda^{(n+1)}_k\right) x_n + \sum_{k=n+1}^\infty \lambda^{(n+1)}_k x_k,\quad n\in \mathbb{N}.
\end{equation}
Henceforth we shall assume that 
\[
\varepsilon<\min\Big(1, \frac{1}{\inf_n\| u_n\|}\Big).
\]
\begin{lem} The following assertions hold for all sequence of scalars $(a_i)\in \coo$:
\begin{itemize}
\item[$(C_1)$]\quad $(z_n)$ is seminormalized and basic. \vskip .2cm 
\item[$(C_2)$]\quad $\displaystyle \Bigg\| \sum_{n=1}^\infty a_i z_i\Bigg\| \leq (2\mathcal{K} + \varepsilon)\Bigg\| \sum_{n=1}^\infty a_n x_n\Bigg\|$. \vskip .2cm 
\item[$(C_3)$]\quad $\displaystyle \Bigg\| \sum_{n=1}^\infty a_i z_i\Bigg\| \geq \frac{(1- \varepsilon)}{4\mathcal{K}}\Bigg\| \sum_{n=1}^\infty a_n x_n\Bigg\|$.\vskip .2cm 
\item[$(C_4)$]\quad $\displaystyle \Bigg\| \sum_{n=1}^\infty a_i z_i -\Big( \sum_{i=1}^\infty a_i \Big)z_1\Bigg\| \geq \frac{1-\varepsilon}{8\mathcal{K}}$.
\end{itemize}
\end{lem}

\begin{proof} $(C_1)$.  It is clear that $(z_n)$ is bounded. Since $(x_n)$ $\lambda$-dominates the summing basis one also has that $\inf_n\| z_n\|>0$. To prove that $(z_n)$ is basic we need an auxiliary sequence $(w_n)_{n\in \mathbb{N}}$ defined as
\[
w_n:=\alpha_n x_n,\quad \alpha_n:= 1 - \sum_{k=n+1}^\infty \lambda^{(n+1)}_k,\quad n\in \mathbb{N}.
\]
It is easy to see that $(w_n)$ is basic with coefficient functionals $(w^*_n)$ satisfying
\[
\sup_n\| w^*_n\|\leq 4\mathcal{K}.
\]
We this in mind we get from (\ref{eqn:1sec3}) that
\[
\sum_n \| w^*_n\| \| z_n - w_n\| \leq \varepsilon \inf_n\| x_n\| <1.
\]
Hence by the Principle of Small Perturbation $(z_n)$ is basic. The proof of $(C_1)$ is complete. 

Fix any sequence $(a_i)$ in $c_{00}$. 

$(C_2)$. A direct calculation using (\ref{eqn:1sec3}) gives
\[
\sup_n| a_n| \sum_{n=1}^\infty \Bigg\| \sum_{k=n+1}^\infty \lambda^{(n+1)}_k x_k\Bigg\|\leq \frac{\varepsilon}{2}\Bigg\| \sum_{n=1}^\infty a_n x_n\Bigg\|.
\]
Thus by Proposition 2.3 in \cite{BF} we deduce 
\[
\Bigg\| \sum_{n=1}^\infty a_n z_n\Bigg\| \leq (2\mathcal{K} + \varepsilon)\Bigg\| \sum_{n=1}^\infty a_n x_n\Bigg\|. 
\]

$(C_3)$. Now as in \cite{BF} we can verify
\[
\sup_n| a_n|\leq \frac{4\mathcal{K}}{\inf_n\| x_n\|} \Bigg\| \sum_{n=1}^\infty a_n \alpha_n x_n\Bigg\|
\]
which when combined with Proposition 2.3 in \cite{BF} yields
\[
\begin{split}
\Bigg\| \sum_{n=1}^\infty a_n z_n\Bigg\| &\geq \Bigg\| \sum_{n=1}^\infty a_n \alpha_n x_n\Bigg\| - \sup_n| a_n| \sup_n\| x_n\| \sum_{n=1}^\infty \sum_{k=n+1}^\infty \lambda^{(n+1)}_k\\
&\geq \frac{(1-\varepsilon)\alpha_1}{2\mathcal{K}}\Bigg\| \sum_{n=1}^\infty a_n x_n \Bigg\|. 
\end{split}
\]

$(C_4)$. Finally, our choices of $u$ and $\lambda_i^{(n)}$'s together $(C_3)$ ensure that 
\begin{equation}\label{eqn:3sec3}
\frac{1-\varepsilon}{4\mathcal{K}} -\lambda \| z_1\|\geq \frac{1-\varepsilon}{8\mathcal{K}}.
\end{equation}
This certainly implies $(C_4)$. The proof of lemma is over. 
\end{proof}

\medskip 
Let us prove now that $K^{+}_\sigma\big( \{ x_n\}\big)$ features the desired properties. 

\begin{lem} The following assertions are valid.
\begin{itemize}
\item[$(C_5)$] There exists $\tau\in \mathcal{T}_w$ so that $K^{+}_\sigma\big( \{ x_n\}\big)$ belongs to $\mathcal{K}_\tau(C)$. 
\item[$(C_6)$] There exists an affine bi-Lipschitz map $f\colon K^{+}_\sigma\big( \{ x_n\}\big)\to K^{+}_\sigma\big( \{ x_n\}\big)$ which is fixed-point free. 
\end{itemize}
\end{lem}

\begin{proof} For simplicity, we set $K=K^{+}_\sigma\big(\{ x_n\} \big)$. Let $\tau=\sigma(X, \Psi_0^{\perp})$ be the locally convex topology induced by $\Psi_0^{\perp}$. Then $\tau \in \mathcal{T}_w$ and $(K,\tau)$ is compact. The former assertion is obvious we only prove the latter. Let $(v_\alpha)$ be a net in $K$ and write 
\[
u_\alpha=\sum_{i=1}^\infty t^\alpha_i x_i\quad \text{ with }\,\, t^\alpha_i\geq 0\,\,\text{ for all }\,\, i, \alpha\,\,\text{ and }\,\, \sum_{i=1}^\infty t^\alpha_i\leq 1. 
\]
The net $\{ t^\alpha\}_\alpha$ defined as $t^\alpha:=( t^\alpha_i)_{i\in \mathbb{N}}$ belongs to $B_{\ell_1}$.  Consider $B_{\ell_1}$ equipped with the weak$^*$ topology $\sigma(\ell_1,c_0)$. Hence it is compact, by Banach-Alaoglu's theorem. After passing to a subnet, if needed, we may suppose then that $t^\alpha\to t$ w.r.t. $\sigma(\ell_1,c_0)$ for some $t=(t_i)_{i\in \mathbb{N}}\in B_{\ell_1}$. Then $t_i\geq 0$ for all $i\in \mathbb{N}$. Let $u:=\sum_{i=1}^\infty t_i x_i$ and note that $u\in K$. Let us prove that $u_\alpha \to u$ w.r.t. $\tau$. To this end, take any $\varphi\in G$ and observe that
\[
\varphi(u)= \sum_{i=1}^\infty t_i \varphi(x_i)\quad \text{ and } \quad \varphi(u_\alpha)= \sum_{i=1}^\infty t^\alpha_i \varphi(x_i)\quad \text{for all } \alpha.
\] 
Since $\varphi\in G$, we know that $\xi:= \big( \varphi(x_i)\big)_{i\in \mathbb{N}}$ belongs to $c_0$. Therefore $\langle t^\alpha, \xi\rangle \to \langle t, \xi\rangle$ implies $\varphi(u_\alpha)\to \varphi(u)$, which proves that $u_\alpha \to u$ w.r.t. $\tau$. It follows that $K\in \mathcal{K}_\tau(C)$. 

It is possible to use the same argument above to prove that $K$ is $\tau$-sequentially compact. This can be done using the fact that $B_{\ell_1}$ is also $\sigma(\ell_1, c_0)$-sequentially compact. The proof of $(C_5)$ is over. 

In order to prove $(C_6)$ let $(z_n)$ be the convex basic sequence of $(x_n)$ defined in (\ref{eqn:zn}), and define $f$ as
\[
f\left( \sum_{i=1}^\infty t_i x_i \right)= \left( 1 - \sum_{i=2}^\infty t_i \right) z_1 + \sum_{i=2}^\infty t_i z_i,\quad x=\sum_{i=1}^\infty t_i x_i \in K. 
\]
It is clear that $f$ is affine and leaves $K$ invariant. We claim that $f$ is fixed point free. Indeed, assume that $x=f(x)$ with $x=\sum_{i=1}^\infty t_i x_i \in K$. This yields a family of equations as follows:\vskip .1cm
\[
\left\{
\begin{split}
t_1&= \left( 1 - \sum_{i=2}^\infty t_i\right)\alpha_1\\
t_2&=\left( 1 - \sum_{i=2}^\infty t_i\right)\lambda^{(2)}_2 + t_2\alpha_2\\
&\vdots\\
t_{n+1}&=\left( 1 - \sum_{i=2}^\infty t_i\right)\lambda^{(2)}_{n+1} + t_2 \lambda^{(3)}_{n+1} +\dots t_n \lambda^{(n+1)}_{n+1} + t_{n+1}\alpha_{n+1},\\
&\vdots
\end{split}
\right.
\]
Summing both sides of these equalities form $1$ to $\infty$, we deduce that $\sum_{i=1}^\infty t_i=1$. Consider the first equality above. Then we have that 
\[
t_1=\left( \sum_{i=1}^\infty t_i - \sum_{i=2}^\infty t_i \right)\alpha_1= t_1 \alpha_1 \quad \Longrightarrow \quad t_1=0. 
\]
Thus $\sum_{i=2}^\infty t_i=1$. This applied in the second equation above implies $t_2=0$, because $0<\alpha_2<1$. Continuing in this way, we arrive inductively  at the conclusion that $t_n=0$ for all $n\in \mathbb{N}$. But this contradicts the equality $\sum_{i=1}^\infty t_i=1$, so $f$ is fixed point free. 

To finish the proof it remains to show that $f$ is bi-Lipschitz. Let $x=\sum_{i=1}^\infty t_i x_i$ and $y=\sum_{i=1}^\infty s_i x_i$ be arbitrary points in $K$. Write $a_i=t_i - s_i$, $i\geq 1$. Clearly 
\[
f(x) - f(y)=\left( -\sum_{i=1}^\infty a_i \right) z_1 + \sum_{i=1}^\infty a_i z_i.
\]
Using $(C_2)$ we deduce
\[
\begin{split}
\| f(x) - f(y)\| &\leq \Bigg| \sum_{i=1}^\infty a_i \Bigg| \| z_1\| + \Bigg\| \sum_{i=1}^\infty a_i z_i\Bigg\|\leq (1+2\mathcal{K}) \| x - y\|,
\end{split}
\]
and from $(C_4)$ we have
\[
\| f(x)- f(y)\| \geq \frac{1-\varepsilon}{4\mathcal{K}}\| x- y\|.
\]
This concludes the proof of lemma and finishes the proof of theorem. 
\end{proof}



\end{document}